\documentclass[11pt,a4paper,twoside]{amsart}

\usepackage[margin=1.4in]{geometry}
\usepackage{amsfonts, amsmath,amssymb}
\usepackage[lite]{amsrefs}
\usepackage{graphicx}
\usepackage{pstricks}
\usepackage{subcaption}
\usepackage{tikz}
\usepackage{hyperref}
\usetikzlibrary{matrix}

\newcommand{\aTop}[2]{\begin{array}{c}{#1}\\{#2}\end{array}}
\newcommand{\op}{\operatorname}
\newcommand{\R}{{\mathbb{R}}}
\newcommand{\Z}{{\mathbb{Z}}}
\newcommand{\N}{{\mathbb{N}}}
\newcommand{\Q}{{\mathbb{Q}}}
\newcommand{\C}{{\mathbb{C}}}
\newcommand{\F}{{\mathbb{F}}}
\newcommand{\ringO}{{\mathcal{O}}}
\newcommand{\Hy}{{\mathcal{H}}^3}
\newcommand{\imQuadRing}{{\mathcal{O}}_{\Q(\sqrt{-m}\thinspace )}}

\newcommand{\SLtwo}{{\rm SL}_2}

\newcommand{\arithGrp}{{\rm SL}_2({\mathbb{Z}}[\sqrt{-2}\thinspace][\frac{1}{2}])}
\newcommand{\BianchiGrp}{{\rm SL}_2({\mathbb{Z}}[\sqrt{-2}\thinspace])}
\newcommand{\congruenceGrp}{{\Gamma_0(\sqrt{-2})}}
\newcommand{\mat}{\begin{pmatrix}a & b\\c & d\end{pmatrix}}
\newcommand{\cohomol}{{\operatorname{H}}}
\newcommand{\ftwo}{{\mathbb F}_2}
\newcommand{\Qe}{{\bf Q}_8}
\newcommand{\Te}{{\bf Te}_{24}}

\newcommand{\dumbbellgraph}{
\begin{pspicture}(0,0.05)(1.0,0.4)
\pscircle(0.2,0.2){0.15}
\psdots(0.35,0.2)
\psline(0.35,0.2)(0.65,0.2)
\psdots(0.65,0.2)
\pscircle(0.8,0.2){0.15}
\end{pspicture} 
}

\newcommand{\decorateddumbbellgraph}{
\scalebox{0.5}{
\begin{pspicture}(-0.5,0.05)(11,4)
\rput(-0.2,2){\huge $\langle C \rangle$}
\pscircle(2,2){1.5}
\rput(2.3,2){\huge $\langle C, B \rangle$}
\psdots[dotsize=12pt](3.5,2)
\psline(3.5,2)(6.5,2)
\rput(5,2.5){\huge $\langle B \rangle$}
\psdots[dotsize=12pt](6.5,2)
\pscircle(8,2){1.5}
\rput(7.7,2){\huge $\langle B, A \rangle$}
\rput(10.2,2){\huge $\langle A \rangle$}
\end{pspicture} 
}
}

\newcommand{\decoratedrhograph}{
\scalebox{0.5}{
\begin{pspicture}(-0.5,0.05)(9,4)
\rput(-0.2,2){\huge $\langle c \rangle$}
\pscircle(2,2){1.5}
\rput(2.4,2){\huge $\langle A, c \rangle$}
\psdots[dotsize=12pt](3.5,2)
\psline(3.5,2)(6.5,2)
\rput(5,2.5){\huge $\langle A \rangle$}
\psdots[dotsize=12pt](6.5,2)
\rput(7.6,2){\huge $\langle A, b \rangle$}
\end{pspicture} 
}
}

\newtheorem{theorem}{Theorem}
\newtheorem{proposition}[theorem]{Proposition}

\newtheorem{lemma}[theorem]{Lemma}
\newtheorem{conjecture}[theorem]{\bfseries Conjecture}

\theoremstyle{definition}
\newtheorem{definition}[theorem]{Definition}

\begin{document}
\title[Verification of the Quillen conjecture over the Bianchi groups]{Verification of the Quillen conjecture \\in the rank 2 imaginary quadratic case}
\author{Bui Anh Tuan}
\email{batuan@hcmus.edu.vn}
\address{Faculty of Mathematics and Computer Science, University of Science, VNU-HCM, 227 Nguyen Van Cu Str., Dist. 5, Ho Chi Minh City, Vietnam.}
\author{Alexander D. Rahm}
\email{Alexander.Rahm@upf.pf}
\address{Laboratoire de math\'ematiques GAATI,
Universit\'e de la Polyn\'esie fran\c{c}aise,
BP 6570,
98702 Faaa,
French Polynesia}
\begin{abstract}
We confirm a conjecture of Quillen in the case of the mod $2$ cohomology of arithmetic groups
${\rm SL}_2(\imQuadRing[\frac{1}{2}])$, where $\imQuadRing$ is an imaginary quadratic ring of integers.
To make explicit the free module structure on the cohomology ring conjectured by Quillen, 
we compute the mod $2$ cohomology of $\arithGrp$ via the amalgamated decomposition of the latter group.
\end{abstract}
\date{\today}
\maketitle

\section*{Introduction}
The Quillen conjecture on the cohomology of arithmetic groups has spur\-red a great deal of mathematics (according to \cite{Knudson}). 
On the cohomology of a linear arithmetic group, Quillen did find a module structure over the Chern classes of the containing group ${\rm GL}_n(\C)$ (equivalently, of ${\rm SL}_n(\C)$). 
Quillen did then conjecture that this module is free~\cite{Quillen}. 
While the conjecture has been proven for large classes of groups of rank 2 matrices, as well as for some groups of rank 3 matrices, 
an obstruction against its validity has been found by Henn, Lannes and Schwartz \cite{HLS},
and this obstruction does occur at least for matrix ranks 14 and higher \cite{HL}. So the scope of the conjecture is not correct in Quillen's original statement, 
and the present paper contributes to efforts on determining the conjecture's correct range of validity. 
Considering the outcomes  of previous research on this question, what seems likely, is that counterexamples may occur already at low matrix rank.
In the present paper however, we confirm the Quillen conjecture over all of the Bianchi groups (the ${\rm SL}_2$ groups over imaginary quadratic rings of integers), 
at the prime number $2$:
\begin{theorem} \label{general theorem}
Let $G$ be a discrete subgroup of ${\rm SL}_2({\mathbb{C}})$, of finite virtual cohomological dimension, and let the $2$-elements-group
$C = \{-1, 1\}$, generated by minus the identity matrix, be contained in $G$.\\
Then $\cohomol^*(G; \thinspace \ftwo)$ is a free module over $\cohomol^*({\rm SL}_2({\mathbb{C}}); \thinspace \ftwo)$.
\end{theorem}
The case that we have in mind is that $G = {\rm SL}_2(\imQuadRing[\frac{1}{2}])$ for $\imQuadRing$ a ring of imaginary quadratic integers.
Theorem~\ref{general theorem}  is a consequence of a theorem of Broto and Henn, as we shall explain in Section~\ref{general proof}.
To make explicit the free module structure on the cohomology ring in one example, we compute the mod $2$ cohomology of $\arithGrp$ via the amalgamated decomposition
$$ \arithGrp \cong \BianchiGrp *_{\congruenceGrp} \BianchiGrp $$
which is known from Serre's classical book~\cite{trees}, and which yields a Mayer--Vietoris long exact sequence on group cohomology that we evaluate.

For ${\rm SL}_2({\mathbb{Z}}[\sqrt{-1}\thinspace][\frac{1}{2}])$, a computation of the mod $2$ cohomology ring structure has already been achieved by Weiss~\cite{Weiss},
but the uniformizing element in the amalgamated decomposition is different for the Gaussian integers ${\mathbb{Z}}[\sqrt{-1}\thinspace]$ from the one for the other imaginary quadratic rings.
And after exclusion of the Gaussian integers, there is a general description of the mod~$2$ cohomology rings of the Bianchi groups and their subgroups~\cites{BerkoveRahm, BLR}.
So for our purposes, the Gaussian integers do not provide a typical example, while the ring ${\mathbb{Z}}[\sqrt{-2}\thinspace]$ does.
Moreover, having confirmed the Quillen conjecture for ${\rm SL}_2$ over the Gaussian integers and over ${\mathbb{Z}}[\sqrt{-2}\thinspace]$ 
removes one of the obstacles against SL$_4(\Z[\frac{1}{2}])$ to satisfy the Quillen conjecture (see diagram (1) in \cite{Weiss}), 
because in the centralizer spectral sequence for SL$_4(\Z[\frac{1}{2}])$ (compare with \cite{Henn}),
the stabilizers  which have the highest complexity are of types $\arithGrp$, ${\rm SL}_2({\mathbb{Z}}[\sqrt{-1}\thinspace][\frac{1}{2}])$ and 
${\rm SL}_2({\mathbb{Z}}[\sqrt{2}\thinspace][\frac{1}{2}])$.

We follow Weiss's strategies for several aspects of our calculation, while we use a different cell complex and recent homological algebra techniques~\cites{BerkoveRahm, BLR}
to overcome specific technical difficulties entering with $\arithGrp$.
Then we arrive at the result stated in Theorem~\ref{cohomology ring} below.
To see how this illustrates the module structure predicted by the Quillen conjecture, let us state the latter over ${\rm SL}_n(\C)$).

\begin{conjecture}[Quillen] \label{Quillen-conjecture}
Let $\ell$ be a prime number. Let $K$ be a number field with
$\zeta_\ell\in K$, and $S$ a finite set of places containing the
infinite places and the places over $\ell$. Then the natural inclusion
$\mathcal{O}_{K,S}\hookrightarrow \mathbb{C}$ makes
$\op{H}^*(\op{SL}_n(\mathcal{O}_{K,S});\thinspace \mathbb{F}_\ell)$ a free
module over the cohomology ring
$\op{H}^*_{\op{cts}}(\op{SL}_n(\mathbb{C});\thinspace \mathbb{F}_\ell)$. 
\end{conjecture} 

\begin{theorem} \label{cohomology ring}
Denote by $e_4$ the image of the second Chern class of the natural representation of ${\rm SL}_2(\C)$,
so  $\ftwo [e_4]$ is the image of $\cohomol^*_{\rm cts}({\rm SL}_2(\C); \ftwo)$.
Then the cohomology ring $\cohomol^*(\arithGrp; \thinspace \ftwo)$ is the free module of rank $10$ over  $\ftwo [e_4]$ with basis
$\{1, x_2,x_3,y_3,z_3,s_3,x_4,s_4,s_5,s_6\},$
where the subscript of the classes specifies their degree as a cohomology class.
\end{theorem}

\subsection*{Organization of the paper}
In Section~\ref{The amalgamated decomposition}, we recall the amalgamated decomposition
$ \arithGrp \cong \BianchiGrp *_{\congruenceGrp} \BianchiGrp $ on which our calculation is based,
as well as its general form which we use in Section~\ref{general proof} in the proof of Theorem~\ref{general theorem}.
In Section~\ref{The cell complex}, we describe a cell complex for the involved congruence subgroup $\congruenceGrp$.
We use it to compute the cohomology of the latter in Section~\ref{The cohomology of the congruence subgroup}.
In Section~\ref{The maps on equivariant spectral sequences},
we determine the maps induced on cohmology by the two injections of $\congruenceGrp$ into $\BianchiGrp$
that characterize the amalgamated product.
Finally in Section~\ref{the proof}, we conclude the proof of Theorem~\ref{cohomology ring}.

\subsection*{Acknowledgements}
This research was funded by Vietnam National University Ho Chi Minh City (VNU-HCM) under grant number C2018-18-02.
The authors are grateful to Oliver Br\"aunling for explaining the involved amalgamated decompositions; 
to Matthias Wendt for helpful discussions;
to Ga\"el Collinet for instigating the present study;
to Grant S. Lakeland for providing a fundamental domain for the relevant congruence subgroup;
and especially to Hans-Werner Henn and the anonymous referee for providing alternative proofs for Theorem~\ref{general theorem}
and for very useful advice on our cohomological calculations.
We also would like to thank the referee for various further constructive suggestions.
Bui Anh Tuan acknowledges the hospitality of University of Luxembourg (Gabor Wiese's grant AMFOR), which he did enjoy thrice for one-month research stays.

\section{The amalgamated decomposition} \label{The amalgamated decomposition}
For a quadratic ring $\ringO_K$, in this section we decompose ${\rm SL}_2(\ringO_K[\frac{1}{2}])$  as a product of two copies of 
${\rm SL}_2(\ringO_K)$ amalgamated over a suitable congruence subgroup.
It is described in Serre's book \textit{Trees} \cite{trees} how to do such a decomposition in general,
 and example decompositions are given for ${\rm SL}_2({\mathbb{Z}}[\frac{1}{p}])$ in Serre's book 
as well as for ${\rm SL}_2({\mathbb{Z}}[\sqrt{-1}\thinspace][\frac{1}{2}])$ in~\cite{Weiss}.

 The following theorem is implicit in the section 1.4 of chapter II of \cite{trees}.

 Let $K$ be a field equipped with a \emph{discrete valuation} $v$.
Recall that $v$ is a homomorphism from $(K \setminus \{0\}, \cdot)$ to $(\Q, +)$ and that
\begin{center}
 $v(x+y) \geq $ min($v(x),v(y)$) for $x,y \in K$,
\end{center}
with the convention that $v(0) = +\infty$.
\\
Let $\ringO$ be the subset of $K$ on which the valuation is non-negative.
Then $\ringO$ is a subring of $K$ and called the Discrete Valuation Ring.
A \emph{uniformizer} is an arbitrary element $\pi \in \ringO \setminus \{0\}$  with $v(\pi) > 0$ 
such that for all $x \in \ringO$ with $v(x)>0$, 
the inequality $v(\pi) \leq v(x)$ holds.
\begin{theorem}[Serre] \label{decomposition into amalgam}
Let $A$ be a dense sub-ring of $K$. Then for any uniformizer $\pi$, we have the decomposition
$$\SLtwo(A)  \cong \SLtwo(\ringO \cap A) \ast_{\Gamma_0(\pi)} \SLtwo(\ringO \cap A), $$
where $\Gamma_0(\pi)$ is the subgroup of $\SLtwo(\ringO \cap A)$ 
of upper triangular matrices modulo the ideal $(\pi)$ of $\ringO \cap A$, injecting
\begin{itemize}
 \item as the natural inclusion into the first factor of type $\SLtwo(\ringO \cap A)$,
 \item via the formula
$$ \mat \mapsto \begin{pmatrix}  \pi^{-1}d\pi & \pi^{-1}c  \\  b\pi & a \end{pmatrix}$$
 into the second factor of type $\SLtwo(\ringO \cap A)$.
\end{itemize}
\end{theorem}
Note that in the above formula, we have replaced Serre's original matrix by an obviously equivalent one suggested by Weiss~\cite{Weiss}.
This has been done with the purpose to have an alternative description of the injection in question as conjugation by the matrix
$ \begin{pmatrix}  0 & 1  \\  \pi & 0 \end{pmatrix}$.

We can now apply Serre's above theorem to quadratic number fields.
\begin{definition}
 Let $K$ be a quadratic number field, let $p$ be a prime number, 
 and let $v_p$ be the $p$-adic valuation of $\Q$.
 We define a function $v_p^{N}$ on $K$ by $$v_p^{N}(x) := v_p(N(x)),$$
where $N(x)$ is the number-theoretic norm of $x \in K$. 
\end{definition}
Note that as $K$ is quadratic, $N(x) = x\overline{x}$ with $\overline{x}$ the Galois conjugate of $x$,
and for $K$ imaginary quadratic, the Galois conjugate is the complex conjugate.

\begin{lemma}[recall of basic algebraic number theory] \label{valuation}
The function $v_p^{N}$
\begin{enumerate}
 \item[(i)] is a valuation of $K$,
 \item[(ii)] extends the valuation $2v_p$ \\
{\rm (the latter being the double of the $p$-adic valuation of $\Q$)},
 \item[(iii)] and takes its values exclusively in $\Z$.
\end{enumerate}
\end{lemma}
An elementary proof can be found in the first preprint version of the present paper.

Now we can turn our attention to ${\rm SL}_2({\mathbb{Z}}[\sqrt{-2}][\frac{1}{2}])$.
 Choose $K := \Q(\sqrt{-2})$, 
 equipped with the valuation $v_2^{N}$ defined above.
Then we can expect $\ringO = \Z_{(2)}[\sqrt{-2}]$. 
Furthermore, we expect $A := \Z[\sqrt{-2}][\frac{1}{2}]$ 
to be dense in $\Q(\sqrt{-2})$ with respect to the topology induced by $v_2^{N}$;
and $\ringO \cap A = \Z[\sqrt{-2}]$.

As $v_2^{N}(\sqrt{-2}) = 1$ 
is minimal on the part of $\ringO$ on which $ v_2^{N}$ is positive,
we can choose $\sqrt{-2}$ as a uniformizer.
Then Theorem~\ref{decomposition into amalgam} yields that
$$\SLtwo( \Z[\sqrt{-2}][\frac{1}{2}])  \cong \SLtwo(\Z[\sqrt{-2}]) \ast_{\Gamma_0(\sqrt{-2})} \SLtwo(\Z[\sqrt{-2}]), $$
where $\Gamma_0(\sqrt{-2})$ is the subgroup of $\SLtwo(\Z[\sqrt{-2}])$ 
of upper triangular matrices modulo the ideal $(\sqrt{-2})$ of $\Z[\sqrt{-2}]$, injecting
\begin{itemize}
 \item as the natural inclusion into the first factor of type $\SLtwo(\Z[\sqrt{-2}])$,
 \item via the formula
$$ \mat \mapsto \begin{pmatrix}  d & {\sqrt{-2}\thinspace\thinspace}^{-1}c  \\  b\sqrt{-2} & a \end{pmatrix}$$
 into the second factor of type $\SLtwo(\Z[\sqrt{-2}])$.
\end{itemize}
We note that $v_2^{N}(2) = 2$, so $2$ is not a uniformizer.

\section{The module over the Chern class ring} \label{general proof}
In this section, we shall explain how Theorem~\ref{general theorem} can be proven with Duflot's ideas~\cite{Duflot},
our access to her ideas being via a theorem of Broto and Henn.
The referee has extracted the crucial argument necessary to use Duflot's ideas for this purpose, and distilled a direct proof. 
We are going to study her/his proof before we explain how to apply Broto and Henn's theorem.
The referee's proof works for $G$ being any subgroup of SL$_2(\C)$ which contains the order-2-subgroup $C := \{-1,1\}$.
We let $i: G \hookrightarrow $SL$_2(\C)$
and $j: C \hookrightarrow G$ denote the inclusions.
We recall from the literature on characteristic classes of vector bundles that
the second Chern class $c_2$ associated to the natural representation of SL$_2(\C)$
yields the cohomology ring structure 
$\cohomol^*({\rm SL}_2(\C);\thinspace \F_2) \cong \F_2[c_2]$, where $c_2$ is of degree~4. 
We recall also that $\cohomol^*(C;\thinspace \F_2) \cong \F_2[t]$, with $t$ of degree $1$, and that $(i \circ j)^*$ is injective
(namely, $(i \circ j)^*(c_2) = t^4$).
Then the referee states the following theorem, and we are going to look into how it is proven.
\begin{theorem}\label{referee}
 $\cohomol^*(G;\thinspace \F_2)$ is a free module over $\cohomol^*({\rm SL}_2(\C);\thinspace \F_2)$.
\end{theorem}
It has been observed in Duflot's work~\cite{Duflot}, 
pursued by Dwyer and Wilkerson~\cite{Dwyer--Wilkerson} and then by Broto and Henn~\cite{BrotoHenn},
that the group morphisms \begin{center}
$(g,c) \mapsto gc $, SL$_2(\C) \times C \to $SL$_2(\C)$ respectively $G \times C \to G$                           
                         \end{center}
provide $\cohomol^*({\rm SL}_2(\C);\thinspace \F_2)$ respectively $\cohomol^*(G;\thinspace \F_2)$ with a comodule structure
with respect to $\cohomol^*(C;\thinspace \F_2)$, given by an induced morphism
$$\delta : \cohomol^*({\rm SL}_2(\C);\thinspace \F_2) \to \cohomol^*({\rm SL}_2(\C);\thinspace \F_2) \otimes \cohomol^*(C;\thinspace \F_2),$$
respectively
$$\delta : \cohomol^*(G;\thinspace \F_2) \to \cohomol^*(G;\thinspace \F_2) \otimes \cohomol^*(C;\thinspace \F_2),$$
such that the induced morphism $i^*$ is a comodule map.
In order to prove that  $\cohomol^*(G;\thinspace \F_2)$ is free, it only remains to prove that multiplication with $i^*(c_2)$ is injective.
We shall do this in the following lemma of the referee.
\begin{lemma}
 If $\alpha \in \cohomol^*(G;\thinspace \F_2)$ is a non-zero element, then $i^*(c_2)\cup \alpha$ is non-zero.
 \end{lemma}
\begin{proof}
 Making use of the above comodule structure, it is of course enough to show that $\delta(i^*(c_2)\cup\alpha)$ is non-zero.
 We easily see that $\delta(i^*(c_2)) = i^*(c_2) \otimes 1 + 1 \otimes t^4$, and therefore
 $$ \delta(i^*(c_2)\cup \alpha) = \delta(i^*(c_2))\cup \delta(\alpha) = (i^*(c_2) \otimes 1 + 1 \otimes t^4)\cup\delta(\alpha).$$
 Let $r$ be the largest integer such that $\delta(\alpha)$ may be written as 
 $$\alpha\otimes 1 +\sum_{|\beta| > |\alpha|-r}\beta\otimes t^{|\alpha|-|\beta|}+\sum_{|\gamma|=|\alpha|-r}\gamma\otimes t^r,$$
 where the last term is non-trivial (it might be equal to $1 \otimes t^{|\alpha|}$), which is possible since $\alpha$ was assumed to be non-zero.
 In the product $(i^*(c_2) \otimes 1 + 1 \otimes t^4)\cup\delta(\alpha)$, the term $\sum_{|\gamma|=|\alpha|-r}\gamma\otimes t^{r+4}$
 cannot cancel, as all other terms are of lower degree on the second factor of the tensor product. This proves the lemma.
\end{proof}
This completes the proof of Theorem~\ref{referee}, so in particular we have the weaker statement presented as Theorem~\ref{general theorem}.

The above proof of Theorem~\ref{general theorem} can be reformulated in the following way using Broto and Henn's theorem.
For this purpose, we define the \textit{depth} of an ideal $I$ in a finitely generated module $M$ 
over a Noetherian ring as in standard commutative algebra textbooks~\cite{Matsumura}:
We assume $I\cdot M \neq M$.
A sequence of elements $x_1, \hdots, x_n$ of $I$ of positive degree is called \textit{regular} on $M$
if $x_1$ is not a zero divisor on $M$ and $x_{i+1}$ is not a zero divisor on the quotient
$M/_{(x_1,\hdots,x_i)M}$.
Under these conditions, any two maximal regular sequences have the same finite length, called the \textit{depth}.

The following lemma is the special case of depth $1$ of a classical commutative algebra result.
The reader can easily work out a proof for it.

\begin{lemma}\label{commutative algebra}
 Let $K$ be a field and $A$ be a positively graduated, connected $K$-algebra. Let $M$ be a graduated $A$-module with $M_n = 0$
 for $n < 0$ and $M_n$ a finite-dimensional $K$-vector space for all $n \in \N \cup \{0\}$.
 Let $x$ be an element in the maximal ideal of $A$ which operates on $M$ injectively by multiplication.
 Then $M$ is a free $K[x]$-module.
\end{lemma}
%

\begin{proof}[Proof of Theorem~\ref{general theorem}]
Theorem 1.1 of \cite{BrotoHenn} states, in its variant provided by remark 2.3 of the same paper: 
\\
\textit{Let $p$ be a prime, $G$ be a discrete group of finite virtual cohomological dimension, 
$X$ a $G$-space and $C$ a central elementary Abelian $p$-subgroup of $G$ acting trivially on $X$.
If $\cohomol^*(X, \thinspace \F_p)$ is finite dimensional over $\F_p$, then the depth of $\cohomol_G^*(X, \thinspace \F_p)$
is at least as big as the rank of $C$.}

We set $p = 2$, and $G$ be a discrete subgroup of SL$_2(\C)$, of finite virtual cohomological dimension, containing $C = \{-1, 1\}$.
Then as we know from Borel and Serre~\cite{BorelSerre}, 
$G$ acts properly discontinuously on the space $X$ constructed as the direct product of ${\rm SL}_2({\mathbb{C}})/{\rm SU}_2$
and a Bruhat-Tits building. 
Note that in the case $G = {\rm SL}_2(\imQuadRing[\frac{1}{2}])$ that we have in mind, 
the Bruhat-Tits building is associated to the $2$-adic group $\SLtwo(\Q(\sqrt{-m})_2)$.
Moreover, $C$ acts trivially on $X$. 
As $X$ is finite-dimensional, the virtual cohomological dimension of the discrete group $G$ is finite.
We consider $\cohomol_G^*(X, \thinspace \F_2)$ as an $\cohomol^*(G, \thinspace \F_2)$-module via the algebra map
$\cohomol^*(G, \thinspace \F_2) \to \cohomol_G^*(X, \thinspace \F_2)$ induced by projection from $X$ to a point.
Then as $\cohomol^*(X, \thinspace \F_2)$ is finite dimensional over $\F_2$,
Broto and Henn's above theorem provides us at least depth $1$ for the maximal ideal in $\cohomol_G^*(X, \thinspace \F_2)$ 
constituted by the elements of strictly positive degree,
and hence a regular sequence $x_1, \hdots, x_n$ with $n \geq 1$, where $x_1 \in \cohomol^*(G, \thinspace \F_2)$ is of strictly positive degree and operates on 
$\cohomol_G^*(X, \thinspace \F_2)$ injectively by multiplication.
Now we can apply Lemma~\ref{commutative algebra} in order to obtain the free module structure on $\cohomol_G^*(X, \thinspace \F_2)$
As $X$ is contractible, $\cohomol^*(G, \thinspace \F_2) = \cohomol_G^*(X, \thinspace \F_2)$.
\end{proof}

\section{The cell complex for the congruence subgroup} \label{The cell complex}
Due to the above amalgamated decomposition, we want to study the action of the congruence subgroup 
$${\Gamma_0(\sqrt{-2})}  := \left\{ \left. \tiny \mat \normalsize \in \BianchiGrp \medspace \right| \medspace c \in \langle \sqrt{-2} \rangle \right\}$$
 in the Bianchi group $\BianchiGrp$ on a suitable cell complex
(a 2-dimensional retract of hyperbolic 3-space). 
For this purpose, we are in the fortunate situation that the symmetric space ${\rm SL}_2({\mathbb{C}})/$SU$_2$ acted on by $\BianchiGrp$
is isometric to real hyperbolic 3-space $\Hy$. We can use the upper half-space model for $\Hy$, where as a set,
$ \Hy = \{ (z,\zeta) \in \C \times \R \medspace | \medspace \zeta > 0 \}. $
Then we can use Poincar\'e's explicit formulas for the action: For $\gamma = \scriptsize \mat \normalsize \in \mathrm{GL}_2(\C)$, 
the action of $\gamma$ on $\Hy$ is given by $\gamma \cdot (z,\zeta) = (z',\zeta')$, where
$$ z' = \frac{\left(\thinspace\overline{cz+d}\thinspace\right)(az+b) +\zeta^2\bar{c}a}{|cz+d|^2 +\zeta^2|c|^2},
\qquad \zeta' = \frac{|\det \gamma|\zeta}{|cz+d|^2 +\zeta^2|c|^2}.$$
Luigi Bianchi~\cite{Bianchi} has constructed a fundamental polyhedron for the action on $\Hy$ of $\BianchiGrp$,
and we use Lakeland's method~\cite{BLR}*{Section 6} to find a set translates of it under elements of $\BianchiGrp$  outside ${\Gamma_0(\sqrt{-2})}$,
which constitute a fundamental domain $\mathcal{F}$ for ${\Gamma_0(\sqrt{-2})}$, strict in its interior:
No two points in its interior can be identified by the action of an element of ${\Gamma_0(\sqrt{-2})}$.
A cumbersome aspect of Bianchi's fundamental polyhedron, and hence also of $\mathcal{F}$, is that it is not compact, but open at cusps which sit in the boundary of $\Hy$.
We shall remedy this aspect with an $\BianchiGrp$-equivariant retraction, namely a retraction of $\Hy$ onto a $2$-dimensional cell complex, along geodesic arcs away from the cusps,
which commutes with the $\BianchiGrp$-action.
Our choice of $\mathcal{F}$, and on it Mendoza's $\BianchiGrp$-equivariant retraction~\cite{Mendoza} 
away from all cusps, are depicted in Figure~\ref{fundamental domains}(A).
An alternative choice would have been to use Fl\"oge's $\BianchiGrp$-equivariant retraction~\cite{Floege} away only from cusps in the ${\Gamma_0(\sqrt{-2})}$-orbit of $\infty$,
and then to use a Borel-Serre compactification on the remaining cusps; the outcome of that alternative is shown in~\cite{BLR}*{figure for $\congruenceGrp$}.
In Figure~\ref{fundamental domains}(B), we list the coordinates of the vertices of $\mathcal{F}$,
and in Figure~\ref{fundamental domains}(C), we display the compact fundamental domain for ${\Gamma_0(\sqrt{-2})}$
obtained from $\mathcal{F}$ in the $2$-dimensional retract of $\Hy$. By boundary identifications on the latter compact fundamental domain, 
we obtain the quotient space of the $2$-dimensional retract modulo the ${\Gamma_0(\sqrt{-2})}$-action, as drawn in Figure~\ref{fundamental domains}(D).
\begin{figure}
\mbox{
\begin{subfigure}[b]{0.6\textwidth}
\caption{Fundamental polyhedron for the action of $\Gamma_0(\sqrt{-2})$ on hyperbolic 3-space.
We retract it $\BianchiGrp$-equivariantly away from the cusps at $0$ and $\infty$, along the dotted edges.}
\label{fundamental polyhedron}
  \includegraphics[width=7cm]{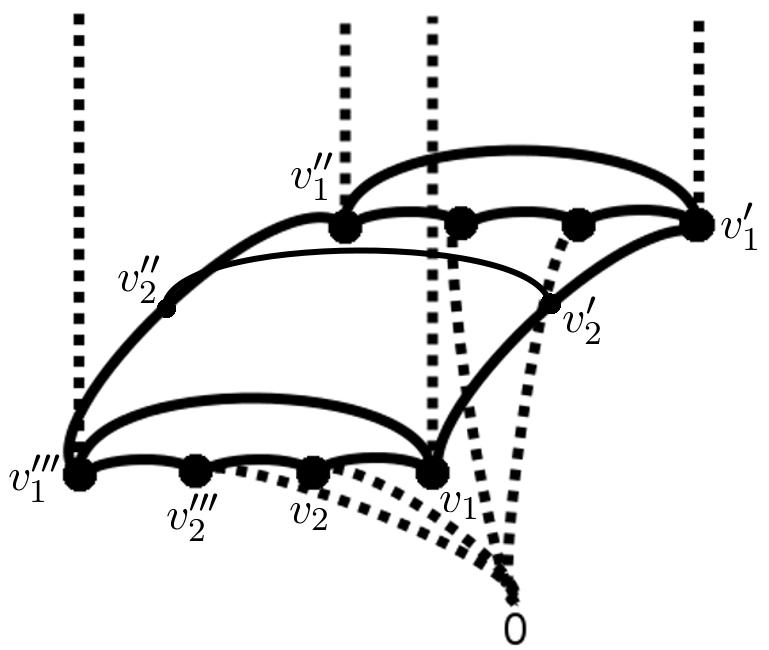}
    \end{subfigure}
     \begin{subfigure}[b]{0.4\textwidth}
\caption{Coordinates of the above vertices in upper half-space. 
Denote $\sqrt{-2}$ by $\omega$, denote the height square by $\zeta^2$ 
and project to the boundary plane at height $\zeta = 0$.}
$\begin{array}{|c|c|c|}
   \hline
   \text{Vertex} & \text{Projection} \medspace z & \zeta^2 \\
   \hline &&\\
   v_1	& -\frac{1}{2}-\frac{\omega}{2} &	1/4 \\
   v_1' &  \frac{1}{2}-\frac{\omega}{2} &	1/4 \\   
   v_1''& \frac{1}{2}+\frac{\omega}{2} &	1/4 \\
   v_1'''& -\frac{1}{2}+\frac{\omega}{2} &	1/4 \\   
   v_2  & -\frac{1}{2}-\frac{\omega}{4} &	1/8 \\
   v_2'  & -\frac{\omega}{2} 		&	1/2 \\
   v_2''  & \frac{\omega}{2} 		&	1/2 \\   
   v_2'''  & -\frac{1}{2}+\frac{\omega}{4} &	1/8 \\
   \hline
 \end{array}$ 
 \end{subfigure}
}
\mbox{
\begin{subfigure}[b]{0.5\textwidth}
\caption{A fundamental domain (strict in its interior) for the action of $\Gamma_0(\sqrt{-2})$ 
on the $2$-dimensional retract is given by the three quadrangles with marked vertices from Figure~\ref{fundamental domains}(A).}
\label{fundamental domain}
 \scalebox{0.85} 
{
\begin{pspicture}(-0.5,-0.7)(6.5,2.7)
\psframe[linewidth=0.04,dimen=outer](0,0)(6,2)
\psdots[dotsize=0.2](0,0)
\psdots[dotsize=0.2](0,2)
\psdots[dotsize=0.2](2,0)
\psdots[dotsize=0.2](2,2)
\psline[linewidth=0.04](2,0)(2,2)
\psdots[dotsize=0.2](4,0)
\psdots[dotsize=0.2](4,2)
\psdots[dotsize=0.2](6,2)
\psline[linewidth=0.04](4,0)(4,2)
\psdots[dotsize=0.2](6,0)
\rput(-0.3,-0.3){$v_2$}
\rput(-0.3,2.3){$v_2'''$}
\rput(2.0,-0.4){$v_1$}
\rput(4.0,-0.4){$v_2'$}
\rput(4.0,2.4){$v_2''$}
\rput(6.3,-0.3){$v_1'$}
\rput(2.0,2.4){$v_1'''$}
\rput(6.3,2.3){$v_1''$}
\rput(-0.3,1){$\langle C \rangle$}
\rput(1,-0.4){$\langle B \rangle$}
\rput(3,-0.4){$\langle A \rangle$}
\rput(5,-0.4){$\langle A \rangle$}
\rput(4.3,1){$\langle c \rangle$}
\rput(6.4,1){$\langle b \rangle$}
\psline[linewidth=0.04](2,1)(1.9,0.9)
\psline[linewidth=0.04](2,1)(2.1,0.9)
\psline[linewidth=0.04](6,1)(5.9,0.9)
\psline[linewidth=0.04](6,1)(6.1,0.9)
\psline[linewidth=0.04](1,0)(0.9,0.1)
\psline[linewidth=0.04](1,0)(0.9,-0.1)
\psline[linewidth=0.04](1.1,0)(1.0,0.1)
\psline[linewidth=0.04](1.1,0)(1.0,-0.1)
\psline[linewidth=0.04](1,2)(0.9,2.1)
\psline[linewidth=0.04](1,2)(0.9,1.9)
\psline[linewidth=0.04](1.1,2)(1.0,2.1)
\psline[linewidth=0.04](1.1,2)(1.0,1.9)
\psline[linewidth=0.04](3,0)(2.9,0.1)
\psline[linewidth=0.04](3,0)(2.9,-0.1)
\psline[linewidth=0.04](3.1,0)(3.0,0.1)
\psline[linewidth=0.04](3.1,0)(3.0,-0.1)
\psline[linewidth=0.04](3,2)(2.9,2.1)
\psline[linewidth=0.04](3,2)(2.9,1.9)
\psline[linewidth=0.04](3.1,2)(3.0,2.1)
\psline[linewidth=0.04](3.1,2)(3.0,1.9)
\psline[linewidth=0.04](2.9,2)(2.8,2.1)
\psline[linewidth=0.04](2.9,2)(2.8,1.9)
\psline[linewidth=0.04](2.9,0)(2.8,0.1)
\psline[linewidth=0.04](2.9,0)(2.8,-0.1)
\psline[linewidth=0.04](5,0)(4.9,0.1)
\psline[linewidth=0.04](5,0)(4.9,-0.1)
\psline[linewidth=0.04](5.2,0)(5.1,0.1)
\psline[linewidth=0.04](5.2,0)(5.1,-0.1)
\psline[linewidth=0.04](5.2,2)(5.1,2.1)
\psline[linewidth=0.04](5.2,2)(5.1,1.9)
\psline[linewidth=0.04](5.1,0)(5.0,0.1)
\psline[linewidth=0.04](5.1,0)(5.0,-0.1)
\psline[linewidth=0.04](5,2)(4.9,2.1)
\psline[linewidth=0.04](5,2)(4.9,1.9)
\psline[linewidth=0.04](5.1,2)(5.0,2.1)
\psline[linewidth=0.04](5.1,2)(5.0,1.9)
\psline[linewidth=0.04](4.9,2)(4.8,2.1)
\psline[linewidth=0.04](4.9,2)(4.8,1.9)
\psline[linewidth=0.04](4.9,0)(4.8,0.1)
\psline[linewidth=0.04](4.9,0)(4.8,-0.1)
\end{pspicture}
}
    \end{subfigure}
    \quad
    \begin{subfigure}[b]{0.45\textwidth}
\caption{Quotient space of the latter fundamental domain by its edge identifications.}
 \includegraphics[height=4cm]{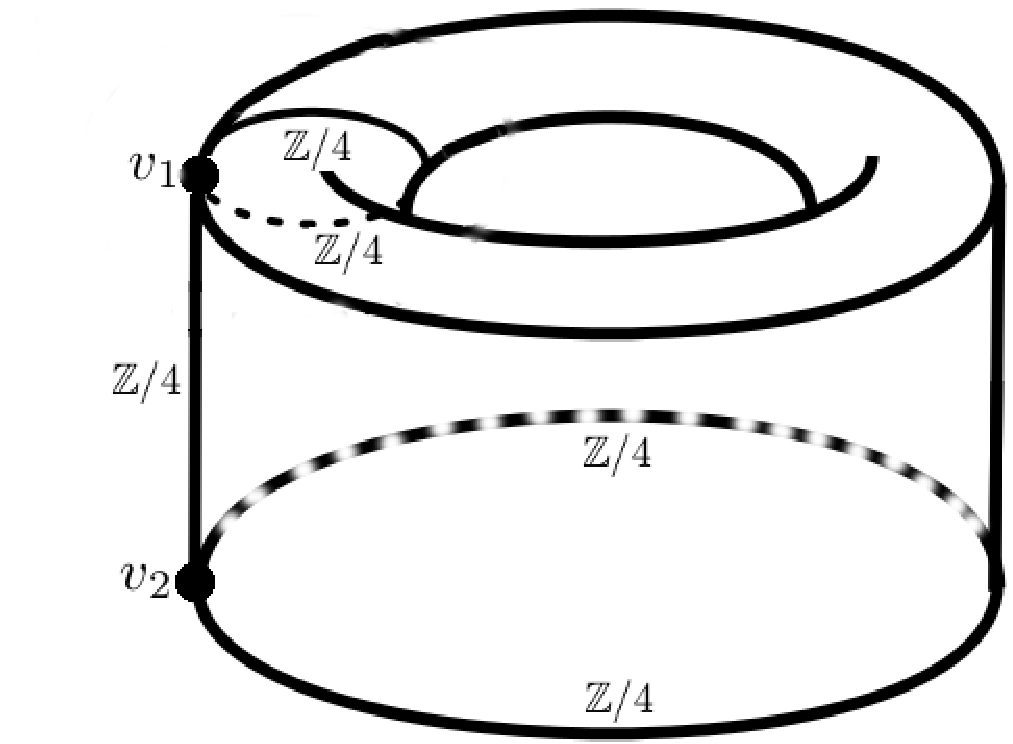}
 \end{subfigure}
 }
 \caption{The cell complex for $\congruenceGrp$}
 \label{fundamental domains}
\end{figure}
The fundamental domain in Figure~\ref{fundamental domains}(C) is subject to edge identifications $\gg$, $\ggg$ and \includegraphics[height=0.3cm]{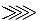} carried out by 
$\begin{pmatrix}
 1 & 0 \\ -\sqrt{-2} & 1
\end{pmatrix}$, and $>$ carried out by $\begin{pmatrix}
 1 & 1 \\ 0 & 1
\end{pmatrix}$, both of which are in $\congruenceGrp$.
With the notation $\omega := \sqrt{-2}$, 
the non-trivial edge stabilizers are generated by the order-4-matrices
$$A = \begin{pmatrix}
 1 & \omega \\ \omega & -1
\end{pmatrix}, \quad
B = \begin{pmatrix}
 -1-\omega & -\omega \\ 2 & 1+\omega
\end{pmatrix}, \quad
C = \begin{pmatrix}
 -1 & -1 \\ 2 & 1
\end{pmatrix},$$
respectively their conjugates by the above mentioned edge identifications.
A fundamental domain for $\BianchiGrp$ is given by the quadrangle $(v_2', v_1', v_1'', v_2'')$.
In $\BianchiGrp$, there are two additional edge stabilizer generators, 
\begin{center}
 $b = \begin{pmatrix}
 1 & -1 \\ 1 & 0 \end{pmatrix}, 
$ of order 6, and $
c = \begin{pmatrix}
0 & -1 \\ 1 & 0
\end{pmatrix}$, of order 4,
\end{center}
which are not in $\congruenceGrp$.

\subsection{The cohomology of the congruence subgroup} \label{The cohomology of the congruence subgroup}
From Figure~\ref{fundamental domains}(D), 
we see that the orbit space of the 2-dimensional retract $X$ of hyperbolic space
under the action of $\congruenceGrp$ has the homotopy type of a 2-torus, so
\begin{center} $\dim_{\ftwo} \cohomol^p(_{\congruenceGrp} \backslash X; \thinspace \ftwo) = $
\scriptsize $\begin{cases}
								    0, & p > 2,\\
                                                                    1, & p = 2,\\
                                                                    2, & p = 1,\\
                                                                    1, & p = 0.\\
                                                                   \end{cases}
$\normalsize \end{center}
We also see in Figure~\ref{fundamental domains}(D) that the 
{non-central $2$-torsion subcomplex} $X_s$, 
namely the union of the cells of $X$ whose cell stabilizers in~$\congruenceGrp$  
contain elements of order a power of $2$, and which are not in the center of~$\congruenceGrp$,
has an orbit space $_\congruenceGrp \backslash X_s$ of shape $\dumbbellgraph$.
Following~\cites{BerkoveRahm, BLR}, we define the co-rank $c$ 
to be the rank of the cokernel of 
$$\cohomol^{1} (_\congruenceGrp \backslash X; \thinspace \ftwo) 
\rightarrow \cohomol^{1} (_\congruenceGrp \backslash X_s; \thinspace \ftwo) $$
induced by the inclusion $X_s \subset X$. 
Again inspecting Figure~\ref{fundamental domains}(D), we can see that the co-rank $c$ vanishes.
Concerning the equivariant spectral sequence discussed in Section~\ref{The maps on equivariant spectral sequences} below,
the $d_2^{p,2q}$ differentials for $p > 0$ are trivial for degree reasons.
Furthermore, we obtain the vanishing of the $d_2^{0,q}$ differentials from lemmas in~\cite{BLR}:
lemma~19 for $d_2^{0,4q+2}$, lemma 21 for $d_2^{0,4q}$ and lemma 25 for $d_2^{0,4q+3}$.
Grant S. \mbox{Lakeland} did use classical  group-geometric  methods (see for instance~\cite{Macbeath}) 
to compute for us from the information presented in Figure~\ref{fundamental domains}
a presentation for $\congruenceGrp$,
$$\congruenceGrp = \langle \alpha, \beta, \gamma, \delta \thinspace | \thinspace \alpha\cdot \beta\cdot \alpha^{-1}\cdot \beta^{-1}, \gamma\cdot \delta\cdot \gamma^{-1}\cdot \delta^{-1}, (\alpha\cdot \delta)^2, (\beta\cdot \gamma)^2, (\alpha\cdot \gamma\cdot \alpha^{-1}\cdot \gamma^{-1})^2\rangle,$$
where 
$\alpha = \begin{pmatrix}
           1& 1 \\ 0 & 1\\
          \end{pmatrix}$,
$\beta = \begin{pmatrix}
           1& \sqrt{-2} \\ 0 & 1\\
          \end{pmatrix}$,
          $\gamma = \begin{pmatrix}1 & 0 \\ -\sqrt{-2} & 1\\
                    \end{pmatrix},
\delta = \begin{pmatrix}
           1 & 0 \\ -2 & 1\\
         \end{pmatrix}$.
So we obtain the Abelianization $\Z^2 \oplus (\Z/2\Z)^2 \cong (\congruenceGrp)^\text{ab} \cong 
\cohomol_1(\congruenceGrp ; \thinspace \Z)$, 
which we insert into the d\'evissage of the equivariant spectral sequence to conclude that the 
$d_2^{0,4q+1}$-differentials vanish as well.
Then using~\cite{BLR}*{theorem 1}, we obtain the following result.
\begin{proposition}
 \begin{center} $\dim_{\ftwo} \cohomol^p(\congruenceGrp; \thinspace \ftwo) = $
\scriptsize $\begin{cases}
								    5, & p \equiv 4 \text{ or }5\mod 4 \text{ and }p>1,\\
                                                                    6, & p \equiv 2 \text{ or }3\mod 4,\\
                                                                    4, & p = 1,\\
                                                                    1, & p = 0.\\
                                                                   \end{cases}
$\normalsize \end{center}
\end{proposition}

\section{The maps on equivariant spectral sequences} \label{The maps on equivariant spectral sequences}
In order to evaluate the Mayer--Vietoris long exact sequence of the amalgamated decomposition 
$ \arithGrp \cong \BianchiGrp *_{\congruenceGrp} \BianchiGrp $, 
we need to find out what the two injections of $\congruenceGrp$ into $\BianchiGrp$,
namely $i$ (the natural inclusion) 
and $j$ (the conjugation map defined in Theorem~\ref{decomposition into amalgam})
induce on the mod $2$ cohomology rings of these groups.
For this purpose, we make use of the equivariant spectral sequences
\begin{center}
$E_1^{p,q} = \bigoplus_{\sigma \in \text{representatives}(_\Gamma \backslash X^p)} \cohomol^q(\Gamma_\sigma ; \thinspace \ftwo)$
converging to $\cohomol^{p+q}(\Gamma ; \thinspace \ftwo)$,
\end{center}
for $\Gamma$ being either $\congruenceGrp$ or $\BianchiGrp$ 
(cf. Brown's book~\cite{Brown}*{chapter VII} for the construction of equivariant spectral sequences).
The notation $\Gamma_\sigma$ stands for the stabilizer in $\Gamma$ of the $p$-cell $\sigma$,
and the $p\text{-cells}$ indexing the above direct sum run through a set of orbit representatives.
As we let the $p\text{-cells}$ for $\BianchiGrp$ run through a fundamental domain contained in the one for $\congruenceGrp$, 
we get two compatible equivariant spectral sequences, and can compute the maps on $\cohomol^{p+q}(\Gamma ; \thinspace \ftwo)$
from the maps between the two $E_1$-pages.
In Section~\ref{The cohomology of the congruence subgroup}, 
we have seen that the $d_2$-differentials are all trivial for $\congruenceGrp$, and having in mind the cylindrical shape of 
$_{\BianchiGrp} \backslash X$ obtained from the edge identification on the quadrangle 
$(v_2', v_1', v_1'', v_2'')$, they are trivial as well for $\BianchiGrp$.
Therefore, this computation splits into two parts: 
firstly, the map on bottom rows 
$E_2^{p,0} \cong \cohomol^p(_\Gamma \backslash X; \thinspace \ftwo)$,
where $X$ is the 2-dimensional retract of hyperbolic space;
and secondly the maps supported on cells with 2-torsion in their stabilizers.
We can use the tools developed in~\cites{BerkoveRahm, BLR} to take advantage of this splitting.
The result of the first part is the following proposition.
\begin{proposition} \label{quotient spaces}
 The injections $i$ and $j$ induce a map
 $$\bigoplus_2 \cohomol^*(_{\BianchiGrp} \backslash X; \thinspace \ftwo)
 \aTop{(i^*,j^*)}{\longrightarrow}
 \cohomol^*(_{\congruenceGrp} \backslash X; \thinspace \ftwo)$$ 
 which is injective, respectively surjective, except for
 $\ker(i^0,j^0) \cong \ftwo$ and \\
 $\operatorname{coker}(i^2,j^2) \cong \ftwo$.
\end{proposition}
\begin{proof}
 The two generators of $\cohomol^1(_{\congruenceGrp} \backslash X; \thinspace \ftwo)$
 are supported on the loops obtained from the edges stabilized by to the matrices $A$, 
 respectively $C$.
 Considering again the cylindrical shape of 
$_{\BianchiGrp} \backslash X$ obtained from the edge identification on the quadrangle 
$(v_2', v_1', v_1'', v_2'')$,
the generator of $\cohomol^1(_{\BianchiGrp} \backslash X; \thinspace \ftwo)$
is supported on the loop obtained from the edge stabilized by the matrix $c$.
The matrices $c$ and $C = i(C)$ are conjugate in $\BianchiGrp$ via the matrix 
$h := \begin{pmatrix}
  1 & 0 \\ 1 & 1
 \end{pmatrix}$
which sends the quadrangle 
$(v_2, v_1, v_1''', v_2''')$ to the quadrangle 
$(v_2', v_1', v_1'', v_2'')$.
Also, the matrix $j(A) \sim C$ is conjugate to $c$ via $h$.
Hence $(i^1,j^1)$ identifies the two loops of $\bigoplus_2 \cohomol^1(_{\BianchiGrp} \backslash X; \thinspace \ftwo)$
with those of $\cohomol^1(_{\congruenceGrp} \backslash X; \thinspace \ftwo)$,
and thus is an isomorphism. The ranks of $(i^0,j^0)$ and $(i^2,j^2)$ are obvious.
\end{proof}
As $X$ is a $2$-dimensional retract, the  $E_1^{p,q}$ terms are concentrated in the three columns $p \in \{0, 1, 2\}$.
And for $E_2^{p,q}({\BianchiGrp} , X; \thinspace \ftwo)$ 
(we shall use this notation with arguments to distinguish between the two spectral sequences),
they are concentrated in the two columns $p \in \{0, 1\}$, because \mbox{$\cohomol^2(_{\BianchiGrp} \backslash X; \thinspace \ftwo) = 0$}.
As we have seen in Section~\ref{The cohomology of the congruence subgroup}
that $d_2^{0,q}({\congruenceGrp} , X; \thinspace \ftwo) = 0$, the $p =2$ column has for all $q \geq 0$ stationary terms
$E_2^{2,q}({\congruenceGrp} , X; \thinspace \ftwo) \cong \cohomol^2(_{\congruenceGrp} \backslash X; \thinspace \ftwo) \cong \ftwo$
which remain as the $E_\infty^{2,q}$-term. So the following lemma gives us all the information that we need in order to achieve the computation.
\begin{lemma} \label{kernel}
 In degrees $q > 0$, $p \in \{0, 1\}$, the map $$\bigoplus_2 E_2^{p,q}({\BianchiGrp} , X; \thinspace \ftwo)
 \aTop{(i^*,j^*)}{\longrightarrow}
 E_2^{p,q}({\congruenceGrp} , X; \thinspace \ftwo)$$ 
  is surjective, with kernel \[  
\begin{array}{l | cccl}
q = 4k+4 &  \langle e_4^i +e_4^j \rangle  &  0 \\
q = 4k+3 &  \langle b_3^i+x_3^j, b_3^j+x_3^i \rangle  &  \langle e_2^{\langle A \rangle, i} b_1^{i}+ e_2^{\langle A \rangle, j} b_1^{j} \rangle  \\
q = 4k+2 &  0    &  \langle e_2^{\langle A \rangle, i}+ e_2^{\langle A \rangle, j} \rangle \\
q = 4k+1 &  0    &   \langle b_1^{i}+ b_1^{j} \rangle \\
\hline k \in \mathbb{N} \cup \{0\} &  p = 0 & p = 1,
\end{array}
\]
where the subscripts specify the degrees of the cohomology classes,
and the superscripts can be ignored (they only serve for tracking back the origin of a class in the calculation).
Throughout the column $p = 2$, there is a constant term $\ftwo$  in the corresponding cokernel.
\end{lemma}
\begin{proof}
From Proposition~\ref{quotient spaces}, 
we already know the contribution of the orbit spaces: a term of type $\ftwo$ 
in rows $q \equiv 0 \mod 4$ in the column $p = 0$,
and a constant term $\ftwo$ throughout the column $p = 2$ in the corresponding cokernel.
Complementary to this, there is a contribution of the non-central 2-torsion subcomplexes $X_s(\Gamma)$,
which we are now going to determine, implicitly making use of information gathered in the proof of Proposition~\ref{quotient spaces},
like the action of the matrix $h$.
The maps $i$ and $j$ from the stabilizers on $X_s(\congruenceGrp)$ to the stabilizers on $X_s(\BianchiGrp)$,
$$\decorateddumbbellgraph \quad \aTop{\longrightarrow}{\aTop{i,j}{\longrightarrow}} \quad \decoratedrhograph$$
are determined by the following conjugacies 
in $\BianchiGrp$:
 \begin{center}
 $j(A) = \begin{pmatrix} -1 & 1 \\ -2 & 1 \end{pmatrix} = P^{-1} \cdot C \cdot P = P^{-1} \cdot i(C) \cdot P$,  
 where $P = \begin{pmatrix} 1+\omega & -1 \\ -2 & 1 -\omega \end{pmatrix} $;
 \\
 \qquad $j(B) = \begin{pmatrix} 1+\omega & -\omega \\ 2 & -1-\omega \end{pmatrix} 
 = Q^{-1} \cdot B \cdot Q = Q^{-1} \cdot i(B) \cdot Q,$ 
 \\
 where
 $Q = \begin{pmatrix} 1 & -1-\omega \\ 0 & 1 \end{pmatrix} $.
 \end{center}
In order to untangle the computation below, about which maps are induced by $i$ and $j$ on the mod $2$ cohomology rings of the finite stabilizers,
we single out already now those which induce zero maps on the $E_2$-level.
Namely, the assignments 
\begin{center}
$j(C) = \begin{pmatrix} 1 & -\omega \\ -\omega & -1 \end{pmatrix} = R^{-1} \cdot A \cdot R$, 
where $R = \begin{pmatrix} 0 & -1 \\ 1 & -\omega \end{pmatrix}$,
respectively $i(A) = A$, 
\end{center}
 induce on the 
$E_2^{1,q}(\congruenceGrp, X, \ftwo)$ terms, 
which have their generators suppported on the loops obtained from the edges stabilized by $C$, respectively $A$,
zero maps coming from the edge stabilized by $A$ in $X_s(\BianchiGrp)$, because the latter supports the zero class in 
$E_2^{1,q}(\BianchiGrp, X, \ftwo)$.
Therefore, we can ignore those two assignments, and work only with the remaining ones.

Having singled out where maps pass to zero on the $E_2$ level, we simply mark ``zero map induced'' at those places in the following diagrams.
For the map $i$, at the passage from the $2$-torsion subcomplex of $\congruenceGrp$ to the 
$2$-torsion subcomplex of $\BianchiGrp$ when the second loop is being folded down onto the middle edge,
we have remaining

\begin{tikzpicture}[align=right] 
\path
   (0,2) node (C) {$\langle C \rangle \cong \Z/4$}
   (0.2,1) node (s1) {$\cong$}
   (2.5,2) node (CB)  {$\langle C, B \rangle \cong \Qe$}
   (2.7,1) node (s2) {$\cong$}
   (5,2) node (B)  {$\langle B \rangle \cong \Z/4$}
   (5.2,1) node (s3) {$\cong$}   
   (7.5,2) node (BA)  {$\langle B, A \rangle \cong \Qe$}
   (8.7,1) node (s4) {\footnotesize$\aTop{A \mapsto A,}{B \mapsto c^{-1}\cdot B\cdot c}$\normalsize}   
   (11,2) node (A)  {$\langle A \rangle \cong \Z/4$}
      (11.2,1) node (zero) {$\aTop{\text{zero map}}{\text{induced}}$}  
   (0,0) node (c) {$\langle c \rangle \cong \Z/4$}
   (2.5,0) node (cA) {$\langle A, c \rangle \cong \Qe$}
   (5,0) node (a)  {$\langle A \rangle \cong \Z/4$}
   (7.5,0) node (Ab) {$\langle A, b \rangle \cong \Te$};
  \draw[->] (C) to node {} (c);
  \draw[->] (CB) to node {} (cA);
    \draw[->] (B) to node {} (a);
    \draw[->] (BA) to node {} (Ab);  
\end{tikzpicture}

On mod $2$ cohomology rings, this induces

\scalebox{0.85}{
\begin{tikzpicture}[align=right] 
\path
   (0,2) node (C) {$\ftwo [e_2](b_1)$}
   (0.2,1) node (s1) {$\cong$}   
   (3,2) node (CB)  {$\ftwo [e_4](x_1,y_1,x_2,y_2,x_3)$}
   (3.2,1) node (s2) {$\cong$}   
   (6,2) node (B)  {$\ftwo [e_2](b_1)$}
   (6.2,1) node (s3) {$\cong$}   
   (9,2) node (BA)  {$\ftwo [e_4](x_1,y_1,x_2,y_2,x_3)$}
   (9.9,1) node (s4) {$\aTop{b_3 \mapsto x_3,}{e_4 \mapsto e_4}$} 
   (12,2) node (A)  {$\ftwo [e_2](b_1)$}
   (12.2,1) node (zero) {$\aTop{\text{zero map}}{\text{induced}}$}  
   (0,0) node (c) {$\ftwo [e_2](b_1)$}
   (3,0) node (cA) {$\ftwo [e_4](x_1,y_1,x_2,y_2,x_3)$}
   (6,0) node (a)  {$\ftwo [e_2](b_1)$}
   (9,0) node (Ab) {$\ftwo [e_4](b_3)$};
  \draw[<-] (C) to node {} (c);
  \draw[<-] (CB) to node {} (cA);
    \draw[<-] (B) to node {} (a);
    \draw[<-] (BA) to node {} (Ab);  
\end{tikzpicture}
}

For the map $j$, we have remaining

\begin{tikzpicture}[align=right] 
\path
   (0,2) node (C) {$\langle C \rangle \cong \Z/4$}
    (0.2,1.25) node (zero) {$\aTop{\text{zero map}}{\text{induced}}$} 
   (1.5,0.5) node (s1) {$\cong$}
   (2.5,2) node (CB)  {$\langle C, B \rangle \cong \Qe$}
   (4.2,0.5) node (s2) {$\cong$}
   (5,2) node (B)  {$\langle B \rangle \cong \Z/4$}
   (4.8,0.5) node (s3) {$\cong$}   
   (7.5,2) node (BA)  {$\langle B, A \rangle \cong \Qe$}
   (7.9,0.65) node (s4) {$\aTop{C \mapsto A,}{B \mapsto b^{-1}\cdot A\cdot b}$}   
   (10,2) node (A)  {$\langle A \rangle \cong \Z/4$}
   (0,0) node (c) {$\langle c \rangle \cong \Z/4$}
   (2.5,0) node (cA) {$\langle A, c \rangle \cong \Qe$}
   (5,0) node (a)  {$\langle A \rangle \cong \Z/4$}
   (7.5,0) node (Ab) {$\langle A, b \rangle \cong \Te$};
  \draw[->] (A) to node {} (c);
  \draw[->] (BA) to node {} (cA);
    \draw[->] (B) to node {} (a);
    \draw[->] (CB.220) to node {} (Ab.160);  
\end{tikzpicture}

On mod $2$ cohomology rings, this induces

\scalebox{0.8}{
\begin{tikzpicture}[align=right] 
\path
   (0,2) node (C) {$\ftwo [e_2](b_1)$}
       (0.2,1.25) node (zero) {$\aTop{\text{zero map}}{\text{induced}}$} 
      (1.5,0.5) node (s1) {$\cong$}
   (3,2) node (CB)  {$\ftwo [e_4](x_1,y_1,x_2,y_2,x_3)$}
         (5.05,0.5) node (s1) {$\cong$}
   (6,2) node (B)  {$\ftwo [e_2](b_1)$}
         (5.8,0.5) node (s1) {$\cong$}
   (9,2) node (BA)  {$\ftwo [e_4](x_1,y_1,x_2,y_2,x_3)$}
   (12,2) node (A)  {$\ftwo [e_2](b_1)$}
   (0,0) node (c) {$\ftwo [e_2](b_1)$}
   (3,0) node (cA) {$\ftwo [e_4](x_1,y_1,x_2,y_2,x_3)$}
   (6,0) node (a)  {$\ftwo [e_2](b_1)$}
      (8.7,0.8) node (s1) {$\aTop{b_3 \mapsto x_3,}{e_4 \mapsto e_4}$}
   (9,0) node (Ab) {$\ftwo [e_4](b_3)$};
  \draw[<-] (A) to node {} (c);
  \draw[<-] (BA) to node {} (cA);
    \draw[<-] (B) to node {} (a);
    \draw[<-] (CB.220) to node {} (Ab.160);  
\end{tikzpicture}
}

Assembling the maps $i$ and $j$ to $(i,j)$, we can now see that $(i^*,j^*)$ is surjective on the $E_1^{p,q}$ terms 
in the two columns $p \in \{0, 1\}$.
To compute the kernel of  $(i^*,j^*)$, we compare the $E_2$ pages: 

The  $E_2$ page for the action of $\BianchiGrp$ on $X$ is concentrated in the two columns $p \in \{0,1\}$, with the following generators,
where the superscripts specify the stabilizer of the supporting cell.
\[  
\begin{array}{l | cccl}
q = 4k+4 &  \langle e_4^{\langle A, c \rangle} \rangle  &  \langle (e_2^{\langle c \rangle})^2 \rangle \\
q = 4k+3 &  \langle  x_3^{\langle A, c \rangle}, b_3^{\langle A, b \rangle} \rangle  &  \langle e_2^{\langle c \rangle} b_1^{\langle c \rangle},  e_2^{\langle A \rangle} b_1^{\langle A \rangle} \rangle  \\
q = 4k+2 &   \langle  x_2^{\langle A, c \rangle}, y_2^{\langle A, c \rangle} \rangle & \langle e_2^{\langle c \rangle}, e_2^{\langle A \rangle} \rangle \\
q = 4k+1 &   \langle  x_1^{\langle A, c \rangle} \rangle &  \langle b_1^{\langle c \rangle} \rangle \\
\hline k \in \mathbb{N} \cup \{0\} &  p = 0 & p = 1
\end{array}
\]
The  $E_2$ page for the action of $\congruenceGrp$ on $X$ is concentrated in the three columns $p \in \{0,1,2\}$, with the following generators.
\small
\[  
\begin{array}{l | cccl}
q = 4k+4 &  \langle e_4^{\langle C,B \rangle}+ e_4^{\langle B,A \rangle} \rangle  &  \langle (e_2^{\langle C \rangle})^2 , (e_2^{\langle A \rangle})^2 \rangle & \ftwo \\
q = 4k+3 &  \langle  x_3^{\langle C, B \rangle}, x_3^{\langle B, A \rangle} \rangle  &  \langle e_2^{\langle C \rangle} b_1^{\langle C \rangle}, e_2^{\langle B \rangle} b_1^{\langle B \rangle}, e_2^{\langle A \rangle} b_1^{\langle A \rangle} \rangle  & \ftwo \\
q = 4k+2 &   \langle  x_2^{\langle C,B \rangle}, y_2^{\langle C, B \rangle}, x_2^{\langle B,A \rangle}, y_2^{\langle B, A \rangle}  \rangle & \langle e_2^{\langle C \rangle}, e_2^{\langle B \rangle}, e_2^{\langle A \rangle} \rangle & \ftwo \\
q = 4k+1 &   \langle  x_1^{\langle C, B \rangle}, x_1^{\langle B,A \rangle} \rangle &  \langle b_1^{\langle C \rangle} +b_1^{\langle A \rangle} \rangle & \ftwo \\
\hline k \in \mathbb{N} \cup \{0\} &  p = 0 & p = 1 & p = 2
\end{array}
\]\normalsize
This yields the claimed kernel (adding $i$, respectively $j$ to the superscripts to specify the relevant copy of $E^{p,q}({\BianchiGrp} , X; \thinspace \ftwo)$ for the pre-image),
and also yields the claimed cokernel.
\end{proof}

\section{Investigating the module structure of the cohomology ring}
\label{the proof}
With the above preparation, we will in this section conclude the proof of Theorem~\ref{cohomology ring}. 
Combining Proposition~\ref{quotient spaces} and Lemma~\ref{kernel},
we can see that the injections $i$ and $j$ induce a map on cohomology of groups,
 $$\bigoplus_2 \cohomol^*({\BianchiGrp}; \thinspace \ftwo)
 \aTop{(i^*,j^*)}{\longrightarrow}
 \cohomol^*({\congruenceGrp}; \thinspace \ftwo)$$ 
 which has the following kernel and cokernel dimensions over $\ftwo$.
\[  
\begin{array}{l | cccl}
q = 4k+5 &  0  &  1 \\
q = 4k+4 &  2  &  1 \\
q = 4k+3 &  3  &  1 \\
q = 4k+2 &  1  &  1 \\
q = 1    &  0  &  0  \\
\hline k \in \mathbb{N} \cup \{0\} &  \dim_{\ftwo} \ker(i^*,j^*) & \dim_{\ftwo} \operatorname{coker}(i^*,j^*)
\end{array}
\]
In the Mayer--Vietoris long exact sequence on group cohomology with $\ftwo$-coefficients derived from the amalgamated decomposition
$$\ \arithGrp \cong \BianchiGrp *_{\congruenceGrp} \BianchiGrp $$
with respect to the maps $i$ and $j$, 

\footnotesize
\begin{tikzpicture}[descr/.style={fill=white,inner sep=1.5pt}]
        \matrix (m) [
            matrix of math nodes,
            row sep=1em,
            column sep=3.3em,
            text height=1.5ex, text depth=0.25ex
        ]
        {   &  &\hdots & \cohomol^{n+1}({\arithGrp}) \\
	    & \cohomol^n({\congruenceGrp}) & \bigoplus\limits_2 \cohomol^n({\BianchiGrp}) & \cohomol^n({\arithGrp}) \\
            & \hdots &  &  \\
        };

        \path[overlay,<-, font=\tiny,>=latex]
        (m-1-3) edge (m-1-4) 
        (m-1-4) edge[out=355,in=175]  (m-2-2)
        (m-2-2) edge node[descr,yshift=-1.6ex] {$(i^n,j^n)$} (m-2-3)
        (m-2-3) edge (m-2-4)
        (m-2-4) edge[out=355,in=175]  (m-3-2);
\end{tikzpicture}
\normalsize

the above calculated dimensions yield

\begin{tikzpicture}[descr/.style={fill=white,inner sep=1.5pt}]
        \matrix (m) [
            matrix of math nodes,
            row sep=1em,
            column sep=3.3em,
            text height=1.5ex, text depth=0.25ex
        ]
        {   & 1 & 0 & \dim_{\ftwo}\cohomol^5({\arithGrp}) \\
            & 1 & 2 & \dim_{\ftwo}\cohomol^4({\arithGrp}) \\
	    & 1 & 3 & \dim_{\ftwo}\cohomol^3({\arithGrp}) \\
            & 1 & 1 & \dim_{\ftwo}\cohomol^2({\arithGrp}) \\
            & 0 & 0 & \dim_{\ftwo}\cohomol^1({\arithGrp}) \\
        };

        \path[overlay,<-, font=\tiny,>=latex]
        (m-1-3) edge (m-1-4) 
        (m-1-4) edge[out=355,in=175]  (m-2-2)
        (m-2-3) edge (m-2-4)
        (m-2-4) edge[out=355,in=175]  (m-3-2)
        (m-3-3) edge (m-3-4)
        (m-3-4) edge[out=355,in=175]  (m-4-2)
        (m-4-3) edge (m-4-4)
        (m-4-4) edge[out=355,in=175]  (m-5-2)
        (m-5-3) edge (m-5-4);        
\end{tikzpicture}
\normalsize

We identify $e_4 := e_4^i +e_4^j$ as the class, multiplication by which yields the $4$-periodicity exposed in Lemma~\ref{kernel}. 
Let us set the following names for the remaining classes:
 $x_4 :=  e_2^{\langle A \rangle, i} b_1^{i}+ e_2^{\langle A \rangle, j} b_1^{j}$, \medspace
 $x_3 := b_3^i+x_3^j$ , \medspace
 $y_3 := b_3^j+x_3^i$, \medspace
 \mbox{$z_3 := e_2^{\langle A \rangle, i}+ e_2^{\langle A \rangle, j}$,} \medspace
$x_2 := b_1^{i}+ b_1^{j}$
and $s_{q+2}$ for the image of the generator of $\cohomol^2(_{\congruenceGrp} \backslash X; \thinspace \ftwo)$
that is generating the $E_2^{2,q}$ term of the equivariant spectral sequence for $\congruenceGrp$
in rows $q \in \{1, 2, 3, 4\}$.
\\
Theorem \ref{general theorem} tells us that $\cohomol^*(\arithGrp; \thinspace \ftwo)$ is a free module over  $\ftwo [e_4]$,
hence the above remaining classes, together with $1 \in \cohomol^0(\arithGrp ; \thinspace \F_2)$ constitute a basis $\{1, x_2,x_3,y_3,z_3,s_3,x_4,s_4,s_5,s_6 \}$ for it.
Thus we get the result stated in Theorem~\ref{cohomology ring}.
\\
We can use our result to calculate dimension bounds for 
$\cohomol^q({\rm GL}_2\left({\mathbb{Z}}[\sqrt{-2}\thinspace]\left[\frac{1}{2}\right]\right); \thinspace \ftwo)$,
see~\cite{GL2calculation}.

\end{document}